\documentclass[12pt]{amsart}

\usepackage{verbatim}
\usepackage{amssymb}
\usepackage{amsbsy}
\usepackage{amscd}
\usepackage{amsmath}
\usepackage{amsthm}
\usepackage[mathscr]{eucal}

\newtheorem{Thm}{Theorem}\numberwithin{Thm}{section}
\numberwithin{Def}{section}
\newtheorem{Lem}{Lemma}\numberwithin{Lem}{section}
\numberwithin{Prop}{section}
\numberwithin{Cor}{section}
\numberwithin{Ass}{section}
\theoremstyle{definition}
\numberwithin{Rem}{section}
\numberwithin{Ex}{section}

\newcommand{\bR}{\ensuremath{\mathbb{R}}}

\newcommand{\eps}{\ensuremath{\varepsilon}}

\renewcommand{\tilde}{\widetilde}
\renewcommand{\bar}{\overline}

\newcommand{\cbra}[1]{\left( #1 \right)}
\newcommand{\kbra}[1]{\left\{ #1 \right\}}

\newcommand{\n}{\nonumber}

\numberwithin{equation}{section}

\title[Non-Markov property of eigenvalue processes]
{Non-Markov property of certain eigenvalue processes 
analogous to Dyson's model}

\author[R. Fukushima]{Ryoki Fukushima}
\author[A. Tanida]{Atsushi Tanida}
\author[K. Yano]{Kouji Yano}

\address{
{\rm Ryoki Fukushima}\\
Department of Mathematics, Kyoto University, Kyoto 606-8502, JAPAN}
\email{fukusima@math.kyoto-u.ac.jp}

\address{
{\rm Atsushi Tanida}\\
Department of Mathematics, Kyoto University, Kyoto 606-8502, JAPAN}
\email{tanida@math.kyoto-u.ac.jp}

\address{
{\rm Kouji Yano}\\
Graduate School of Science, Kobe University, Kobe 657-8501, JAPAN} 
\email{kyano@math.kobe-u.ac.jp}

\subjclass[2000]{15A52, 60-06, 60J65, 60J99}
\keywords{non-Markov property, random matrix, eigenvalue process, Dyson's model, beta-ensembles}

\thanks{The research of the third author was supported by KAKENHI (20740060)}

\begin{document}
\begin{abstract}
It is proven that the eigenvalue process of Dyson's random matrix process 
of size two 
becomes non-Markov 
if the common coefficient $1/\sqrt{2}$ in the non-diagonal entries 
is replaced by a different positive number. 
\end{abstract}

\maketitle

\section{Introduction}
Dyson \cite{Dys62} has introduced the matrix-valued stochastic process 
\begin{align}
\Xi(t) = 
\begin{pmatrix}
B_{1,1}(t) & \frac{1}{\sqrt{2}} B_{1,2}(t) & 
\cdots & \frac{1}{\sqrt{2}} B_{1,N}(t) \\
\frac{1}{\sqrt{2}} \bar{B_{1,2}(t)} & B_{2,2}(t) & 
\cdots & \frac{1}{\sqrt{2}} B_{2,N}(t) \\
\vdots & \vdots 
& \ddots & \vdots \\
\frac{1}{\sqrt{2}} \bar{B_{1,N}(t)} & \frac{1}{\sqrt{2}} \bar{B_{2,N}(t)} & 
\cdots & B_{N,N}(t) 
\end{pmatrix}
\n
\end{align}
to model the dynamics of particles with the Coulomb type interactions, 
where $ B_{i,i} $'s are real Brownian motions 
and $ B_{i,j} $'s for $ i<j $ are complex Brownian motions 
all of which are mutually independent. 
He proved that the eigenvalue processes $ \lambda_1,\ldots,\lambda_N $ satisfy 
the (system of) stochastic differential equations 
\begin{align}
d\lambda_i(t) 
= d\beta_i(t) + \frac{\beta}{2} \sum_{j \neq i} \frac{1}{\lambda_i(t) - \lambda_j(t)} dt 
\n
\end{align}
with $ \beta=2 $. 
It has been proven later that if the complex Brownian motions are replaced 
by real or quaternion Brownian motions, the eigenvalue processes satisfy 
similar stochastic differential equations with $\beta=1$ or $4$, respectively. 
(See \cite{Bru91, KT04} for discussions based on the stochastic analysis.) 
These processes are now called Dyson's Brownian motion models for 
GOE, GUE, and GSE when $\beta = 1, 2$, and $4$, respectively. 
In any case, it is remarkable that the process 
$ \Lambda=(\lambda_1,\ldots,\lambda_N) $ is Markov.

We may ask the following question: 
``Does the process $ \Lambda $ remain Markov 
if we replace the common coefficient $ 1/\sqrt{2} $ by a different positive number?" 
In this paper, we give the {\em negative} answer to this question 
when the matrix size $ N=2 $. 

Let $ c \ge 0 $ and $ \delta > 0 $. 
Consider the $ 2 \times 2 $-matrix-valued process 
\begin{align}
\Xi^{c,\delta}(t) = 
\begin{pmatrix}
B_1(t) & \sqrt{c/2} \, \xi^{\delta}(t) \\
\sqrt{c/2} \, \xi^{\delta}(t) & B_2(t) 
\end{pmatrix}
\label{intro: Xi}
\end{align}
where $ B_1 $ and $ B_2 $ are two independent standard Brownian motions 
and $ \xi^{\delta} $ is a Bessel process of dimension $ \delta $ starting from 0 
which is independent of $ B_1 $ and $ B_2 $. 
We see in Lemma \ref{lem: GOE} that $ \Xi^{c,\delta} $ with $ \delta=1,2 $, or 4
is unitarily equivalent in law to 
\begin{align}
\tilde{\Xi}^{c,\delta}(t) = 
\begin{pmatrix}
B_1(t) & \sqrt{c/2} \, B_3(t) \\
\sqrt{c/2} \, \overline{B_3(t)} & B_2(t) 
\end{pmatrix}\label{cDyson}
\end{align}
with $B_3$ a real, complex, or quaternion Brownian motion independent of 
$B_1$ and $B_2$, respectively. 
Let $ \lambda_1(t) $ and $ \lambda_2(t) $ for $ t \ge 0 $ denote 
the eigenvalues of the Hermitian matrix $ \Xi^{c,\delta}(t) $ 
such that $ \lambda_1(t) \ge \lambda_2(t) $. 
Define the two-dimensional process 
$ \Lambda^{c,\delta} = (\lambda_1,\lambda_2) $.

When $ c=0 $, 
$ \lambda_1(t) $ and $ \lambda_2(t) $ are 
nothing but the order statistics 
of $ B_1(t) $ and $B_2(t) $, that is, 
$ \lambda_1(t)=\max \kbra{B_1(t),B_2(t)} $ 
and $ \lambda_2(t)=\min \kbra{B_1(t),B_2(t)} $. 
Hence it is obvious that the process $ \Lambda^{0,\delta} $ is Markov. 

When $ c=1 $, the process \eqref{intro: Xi} is a time-dependent version 
of Dumitriu-Edelman's matrix model for beta-ensembles (cf.~\cite{DE02}) and we see in Lemma \ref{lem: Dyson} that 
the processes $ \lambda_1(t) $ and $ \lambda_2(t) $ satisfy 
Dyson's stochastic differential equations with index $ \beta=\delta $ given by 
\begin{align}
d \lambda_1(t) =& d \beta_1(t) + \frac{\delta}{2(\lambda_1(t)-\lambda_2(t))} dt , 
\label{intro: Dyson1}
\\
d \lambda_2(t) =& d \beta_2(t) + \frac{\delta}{2(\lambda_2(t)-\lambda_1(t))} dt 
\label{intro: Dyson2}
\end{align}
for two independent Brownian motions $ \beta_1(t) $ and $ \beta_2(t) $. 
In particular, the process $ \Lambda^{1,\delta}(t) $ is Markov. 

\begin{Thm} \label{thm: Lambda}
The process $ \Lambda^{c,\delta} $ is Markov 
if and only if $ c \in \{0, 1\} $. 
\end{Thm}

We prove this theorem by reducing it to the following. 

\begin{Thm} \label{thm: Bessel}
Let $ \delta_1,\delta_2>0 $. 
Let $ X^{\delta_1} $ and $ Y^{\delta_2} $ be two independent 
squared Bessel processes starting from 0 
of dimension $ \delta_1 $ and $ \delta_2 $, respectively. 
Then the process $ Z^c(t)=cX^{\delta_1}(t)+Y^{\delta_2}(t) $ for $ c \ge 0 $ is Markov 
if and only if $ c \in \{0, 1\} $. 
\end{Thm}

Theorems \ref{thm: Lambda} and \ref{thm: Bessel} 
seem similar to Matsumoto-Ogura's $ cM-X $ theorem \cite{MO04}. 
Let $ X $ be a Brownian motion 
and set $ M(t)=\sup_{0 \le s \le t} X(s) $. 
When $ c=0,1,2 $, the process $ cM-X $ is Markov; 
indeed, $ -X $ is a Brownian motion, 
$ M-X $ is a reflecting Brownian motion by L\'evy's theorem 
(see, e.g., \cite[Thm.VI.2.3]{RY99}), 
and $ 2M-X $ is a three-dimensional Bessel process by Pitman's theorem 
(see, e.g., \cite[Thm.VI.3.5]{RY99}). 

\begin{Thm}[\cite{MO04}]
The process $ cM-X $ is Markov if and only if $ c \in \{0, 1, 2\} $. 
\end{Thm}

\section{Non-Markov property of the eigenvalue processes}

\begin{proof}[Proof of Theorem \ref{thm: Lambda} provided Theorem \ref{thm: Bessel} is justified]

An elementary calculation shows that $ \lambda_1 $ and $ \lambda_2 $ are given by 
\begin{align*}
\lambda_1(t) =& 
\frac{1}{2} \kbra{ B_1(t) + B_2(t) + \sqrt{ (B_1(t) - B_2(t))^2 + 2c \xi^{\delta}(t)^2 } } , 
\\
\lambda_2(t) =& 
\frac{1}{2} \kbra{ B_1(t) + B_2(t) - \sqrt{ (B_1(t) - B_2(t))^2 + 2c \xi^{\delta}(t)^2 } } . 
\end{align*}
Set $ B_3(t) = \{ B_1(t) + B_2(t) \}/\sqrt{2} $, 
$ X^1(t) = \{ B_1(t) - B_2(t) \}^2/2 $ 
and $ Y^{\delta}(t) = \xi^{\delta}(t)^2 $. 
Then $ B_3 $ is a real Brownian motion, 
$ X^1 $ and $ Y^{\delta} $ are squared Bessel processes of dimension 1 and $ \delta $, 
respectively. 
Moreover, $ B_3 $, $ X^1 $, and $ Y^{\delta} $ are mutually independent. 
It follows that 
\begin{align*}
\lambda_1(t) =& 
\frac{1}{\sqrt{2}} \kbra{ B_3 + \sqrt{ X^1(t) + c Y^{\delta}(t) } } , 
\\
\lambda_2(t) =& 
\frac{1}{\sqrt{2}} \kbra{ B_3 - \sqrt{ X^1(t) + c Y^{\delta}(t) } } . 
\end{align*}

It is obvious that 
the two dimensional process $ \Lambda^{c,\delta} = (\lambda_1,\lambda_2) $ is Markov 
if and only if so is the process $ (\lambda_1+\lambda_2,\lambda_1-\lambda_2) $. 
Since 
\begin{align}
\lambda_1 + \lambda_2 =& \sqrt{2} B_3, 
\label{prf Lambda: 1} \\
\lambda_1-\lambda_2 =& \sqrt{2} \sqrt{ X^1 + c Y^{\delta} } 
\label{prf Lambda: 2}
\end{align}
and they are independent, 
for the process $ \Lambda^{c,\delta} $ to be Markov 
it is necessary and sufficient that 
the process $ X^1 + c Y^{\delta} $ is Markov. 
This is equivalent to $ c=0 $ or $ 1 $ by Theorem \ref{thm: Bessel}. 
\end{proof}

\begin{Lem} \label{lem: Dyson}
For $ c=1 $ and $ \delta > 0 $, 
consider the $ 2 \times 2 $-matrix-valued process $ \Xi^{1,\delta} $ 
defined by \eqref{intro: Xi}. 
Then the corresponding eigenvalue processes satisfy 
the stochastic differential equations 
\eqref{intro: Dyson1}--\eqref{intro: Dyson2}. 
\end{Lem}

\begin{proof}
Set $ \tilde{\lambda} = (\lambda_1-\lambda_2)/\sqrt{2} $. 
Then, by \eqref{prf Lambda: 2} for $ c=1 $ and by Shiga-Watanabe's theorem 
(see, e.g., \cite[Thm.XI.1.2]{RY99}), 
we see that the process $ \tilde{\lambda} $ is a Bessel process of dimension $ 1+\delta $. 
Hence we have 
\begin{align}
d \tilde{\lambda}(t) = 
d B_4(t) + \frac{\delta}{2} \frac{1}{\tilde{\lambda}(t)} dt 
\label{prf Lambda: 3}
\end{align}
where $ B_4 $ is a real Brownian motion independent of $ B_3 $. 
If we set $ \beta_1 = (B_3+B_4)/\sqrt{2} $ and $ \beta_2 = (B_3-B_4)/\sqrt{2} $, 
then $ \beta_1 $ and $ \beta_2 $ are two independent real Brownian motions. 
Therefore, combining \eqref{prf Lambda: 3} with \eqref{prf Lambda: 1}, 
we conclude that 
\eqref{intro: Dyson1}--\eqref{intro: Dyson2} hold. 
\end{proof}

\begin{Lem} \label{lem: GOE}
Let $ c>0 $, $ \delta=1, 2$, or $4$, and 
$\Xi^{c,\delta}$ and $\tilde{\Xi}^{c,\delta}$ be the matrix-valued processes 
defined by \eqref{intro: Xi} and \eqref{cDyson}, respectively.  
Then, there exists a unitary matrix-valued process $U_{\delta}(t)$ such that 
\begin{align*}
\Bigl(\Xi^{c,\delta}(t)\Bigr)_{t \ge 0} \stackrel{\rm law}{=} 
\left(U_{\delta}(t) \tilde{\Xi}^{c,\delta}(t) U_{\delta}^*(t)\right)_{t \ge 0}. 
\end{align*}
In particular, eigenvalue processes associated with 
$\Xi^{c,\delta}$ and $\tilde{\Xi}^{c,\delta}$ have the same law. 
\end{Lem}

\begin{proof}
We define 
\begin{align*}
U_{\delta}(t) = 
\begin{pmatrix}
1 & 0 \\
0 & \frac{B_3(t)}{|B_3(t)|} 
\end{pmatrix}
1_{B_3(t) \neq 0}+ 
\begin{pmatrix}
1 & 0 \\
0 & 1 
\end{pmatrix}
1_{B_3(t) = 0} 
\end{align*}
by using $ B_3 $ in \eqref{cDyson}. 
Then we have  
\begin{align*}
U_{\delta}(t) \tilde{\Xi}^{c,\delta}(t) U_{\delta}^*(t) = 
\begin{pmatrix}
B_1(t) & \sqrt{c/2} \, |B_3(t)| \\ 
\sqrt{c/2} \, |B_3(t)|  & B_2(t) 
\end{pmatrix}, 
\end{align*}
which shows the desired result since $|B_3| \stackrel{\rm law}{=} \xi^{\delta}$. 
\end{proof}

\section{Transition probability density of squared Bessel processes}
In this section, we recall some basic asymptotic estimates on the transition 
probability density $ p_t^{\delta}(x,y) $ of squared Bessel processes 
of dimension $ \delta $ which we shall use later. 
We first note that it has an expression 
\begin{align}
p_t^{\delta}(x,y) = \frac{1}{2t} \cbra{ \frac{y}{x} }^{(\delta-2)/4} 
\exp \cbra{- \frac{x+y}{2t}} I_{(\delta-2)/2} \cbra{\frac{\sqrt{xy}}{t}}
\label{density}  
\end{align}
for $x,y>0$, where $ I_{\nu} $ stands for the modified Bessel function of index $ \nu $ 
(see, e.g., \cite[Cor.XI.1.4]{RY99}). 
Now let us recall following two asymptotic estimates on 
the modified Bessel function (see, e.g., Sect.\ 5.16.4 of \cite{Leb72}): 
\begin{align}
&I_{\nu}(x) \sim \frac{1}{\Gamma(\nu+1) } \Bigl( \frac{x}{2} \Bigr) ^{\nu} 
\quad \text{as $ x \downarrow 0$}, \label{small}\\
&I_{\nu}(x) \sim \frac{e^{x}}{\sqrt{2 \pi x}} 
\quad \text{as $ x \uparrow \infty $}. \label{large}
\end{align}
Here, $ f(x) \sim g(x) $ means $ f(x)/g(x) \to 1 $ 
in the subsequently indicated limit. 

Using \eqref{small} in \eqref{density}, we can derive 
\begin{align}
p_t^{\delta}(0+,y) = 
\frac{y^{(\delta/2)-1}}{(2t)^{\delta/2} \Gamma (\delta/2)} 
\exp \cbra{- \frac{y}{2t}} \label{from0}
\end{align}
for $t, y > 0$ and 
\begin{equation}
\begin{split}
\lim_{y \to 0+} y^{1-\delta/2} p_t^{\delta}(x,y)
& = x^{1-\delta/2} p_t^{\delta}(0+,x)\\
& = \frac{1}{(2t)^{\delta/2} \Gamma (\delta/2)} 
\exp \cbra{- \frac{x}{2t}} \label{to0}  
\end{split}
\end{equation}
for $t, x > 0$. 
On the other hand \eqref{large} together with \eqref{density} yields
\begin{align}
p_t^{\delta}(x,y) \sim \frac{1}{2 t \sqrt{2\pi}} \frac{y^{(\delta-3)/4}}{x^{(\delta-1)/4}} 
\exp \cbra{- \frac{x+y-2\sqrt{xy}}{2t} }\label{far} 
\end{align}
as $ \sqrt{xy} \to \infty $.

\section{Non-Markov property of weighted sums of two independent squared Bessel processes}
For the proof of Theorem \ref{thm: Bessel}, we may restrict ourselves to 
$ 0<c<1 $; otherwise consider $ Z^c/c $ instead. 
We prove that $ Z^c $ is non-Markov by checking that 
the conditional law 
\begin{align}
P \cbra{ Z^c(2) \in dz_3 \mid Z^c(\eps)=z_1 , \ Z^c(1)=z_2 } \quad \text{for }  0<\eps<1   
\label{prf Bes: cond trans prob}
\end{align}
does depend on $ (\eps,z_1) $. 
This conditional law has the density 
\begin{align*}
P \cbra{ Z^c(2) \in dz_3 \mid Z^c(\eps)=z_1 , \ Z^c(1) =z_2 } 
= \frac{q(z_2,z_3;\eps,z_1)}{q(z_2;\eps,z_1)} dz_3, 
\end{align*}
where $q(z_2,z_3;\eps,z_1)$ and $q(z_2;\eps,z_1)$ are the densities of 
the joint laws of $ (Z^c(\eps),Z^c(1),Z^c(2)) $ and $ (Z^c(\eps),Z^c(1)) $, 
respectively. Thus it suffices to prove that the fraction 
${q(z_2,z_3;\eps,z_1)}/{q(z_2;\eps,z_1)}$ depends on $(\eps, z_1)$. 

To this end, we shall use the integral expression 
\begin{align*}
q(z_2,z_3;\eps,z_1) 
&= \int_0^{z_1} dx_1 \int_0^{z_2} dx_2 \int_0^{z_3} dx_3 A_{1,1} A_{1,2} A_{1,3},\\ 
q(z_2;\eps,z_1) 
&= \int_0^{z_1} dx_1 \int_0^{z_2} dx_2 A_{1,1} A_{1,2}, 
\end{align*}
where 
\begin{align*}
A_{1,1} =& p_{\eps}^{\delta_1}(0+,x_1) p_{\eps}^{\delta_2}(0+,z_1-cx_1) , 
\\
A_{1,2} =& p_{1-\eps}^{\delta_1}(x_1,x_2) p_{1-\eps}^{\delta_2}(z_1-cx_1,z_2-cx_2) , 
\\
A_{1,3} =& p_{1}^{\delta_1}(x_2,x_3) p_{1}^{\delta_2}(z_2-cx_2,z_3-cx_3) . 
\end{align*}

We divide the proof into several steps. 
First of all, we prove 

\begin{Lem} \label{prf Bessel: asymp}
Let $ f(\lambda,\cdot) $ for $ \lambda>0 $ be a bounded measurable function on $ (0,1) $. 
Suppose that $ f(\lambda,x/\lambda) $ converges to a constant $ f(\infty ,0) $ for any $ x \in (0,1) $ as $ \lambda \to \infty $. 
Let $ \phi \in C^1((0,1)) $ and suppose that $ \phi(0+)=a \in \bR $, $ \phi'(0+)=b>0 $ 
and $ \phi'(x)>0 $ for $ x \in (0,1) $. 
Let $ \nu > 0 $. 
Then 
\begin{align}
\int_0^1 e^{- \lambda \phi(x)} f(\lambda,x) x^{\nu-1} dx 
\sim f(\infty ,0) \frac{\Gamma (\nu ) }{b^{\nu } } 
\lambda^{- \nu } e^{- a \lambda} 
\quad \text{as $ \lambda \to \infty $}. \label{Lem 0}
\end{align}
\end{Lem}

\begin{proof}
Changing variables to $ u=\lambda x $, we find that the left hand side
of \eqref{Lem 0} equals  
\begin{align*}
\lambda^{-\nu} e^{- a \lambda} 
\int_0^{\lambda} e^{- \lambda \{ \phi(u/\lambda) -a \}} f(\lambda,u/\lambda) du . 
\end{align*}
Note that $ \lambda \{ \phi(u/\lambda) -a \} \ge K u $ 
for $ u \in (0,\lambda) $ and $ \lambda>0 $ 
where $ K = \inf_{x \in (0,1)} \{ \phi(x)-\phi(0+) \}/x > 0 $. 
Hence we see that 
\begin{align*}
\lim_{\lambda \to \infty } 
\int_0^{\lambda} e^{- \lambda \{ \phi(u/\lambda) -a \}} f(\lambda,u/\lambda) du 
= f(\infty ,0) \int_0^{\infty } e^{- bu} u^{\nu -1} du 
\end{align*}
by the dominated convergence theorem. 
\end{proof}

Second, we take the limit as $ \eps \to 0 $. 

\begin{Lem} \label{lem: lemma1}
\begin{align*}
\lim_{\eps \to 0+} \frac{q(z_2,z_3;\eps,z_1)}{q(z_2;\eps,z_1)} 
= \frac{q(z_2,z_3;z_1)}{q(z_2;z_1)} 
\end{align*}
with 
\begin{align*}
q(z_2,z_3;z_1) = \int_0^{z_2} dx_2 \int_0^{z_3} dx_3 A_{2,1} A_{2,2} , 
\quad 
q(z_2;z_1) = \int_0^{z_2} dx_2 A_{2,1} 
\end{align*}
where $ A_{2,2}=A_{1,3} $ and 
\begin{align*}
A_{2,1} = \left. \Big. A_{1,2} \right|_{\eps \to 0+, \, x_1 \to 0+}  = 
p_{1}^{\delta_1}(0+,x_2) p_{1}^{\delta_2}(z_1,z_2-cx_2) . 
\end{align*}
\end{Lem}

\begin{proof}
We know that 
\begin{align*}
A_{1,1} = \frac{(x_1)^{(\delta_1/2)-1}(z_1-cx_1)^{(\delta_2/2)-1}}
{(2\eps)^{(\delta_1+\delta_2)/2} \Gamma (\delta_1/2) \Gamma (\delta_2/2)} 
\exp \cbra{ - \frac{1}{2\eps} \kbra{z_1 + (1-c)x_1} } 
\end{align*}
from \eqref{from0}. 
Now we can rewrite $ q(z_2,z_3;\eps,z_1) / q(z_2;\eps,z_1) $ as $ F_1/G_1 $ with 
\begin{align}
F_1 =& \int_0^{z_1} A_{1,4}(\eps,x_1) x_1^{(\delta_1/2)-1} e^{ - (\tilde{c}/\eps)x_1 } dx_1 
\label{prf Bessel: F3}
\\
G_1 =& \int_0^{z_1} A_{1,5}(\eps,x_1) x_1^{(\delta_1/2)-1} e^{ - (\tilde{c}/\eps)x_1 } dx_1 
\label{prf Bessel: G3}
\end{align}
where $ \tilde{c}=(1-c)/2 $ and 
\begin{align*}
A_{1,4}(\eps,x_1) =& (z_1-cx_1)^{(\delta_2/2)-1} 
\int_0^{z_2} dx_2 \int_0^{z_3} dx_3 A_{1,2} A_{1,3} , 
\\
A_{1,5}(\eps,x_1) =& (z_1-cx_1)^{(\delta_2/2)-1} 
\int_0^{z_2} dx_2 A_{1,2} . 
\end{align*}
Using Lemma \ref{prf Bessel: asymp} 
in the integrals \eqref{prf Bessel: F3} and \eqref{prf Bessel: G3}, we have 
\begin{align*}
F_1 \sim& \eps^{\delta_1/2} \Gamma (\delta_1/2) \tilde{c}^{-\delta_1/2} A_{1,4}(0,0) , 
\\
G_1 \sim& \eps^{\delta_1/2} \Gamma (\delta_1/2) \tilde{c}^{-\delta_1/2} A_{1,5}(0,0) 
\end{align*}
as $ \eps \to 0+ $. 
Here we have used the fact that 
$ A_{1,4}(\eps,x_1) $ and $ A_{1,5}(\eps,x_1) $ are continuous 
in $ \eps \in [0,\infty ) $ and $ x_1 \in [0,z_1] $. 
Therefore, $F_1/G_1$ approaches to $ A_{1,4}(0,0) / A_{1,5}(0,0) = q(z_2,z_3;z_1)/q(z_2;z_1) $. 
\end{proof}

Third, we study the asymptotic behavior of the numerator $ q(z_2,z_3;z_1) $ as $ z_3 \to 0+ $. 

\begin{Lem} \label{lem: lemma2}
\begin{align*}
\lim_{z_3 \to 0+} z_3^{1-(\delta_1+\delta_2)/2} q(z_2,z_3;z_1) = 
C_1 \tilde{q}(z_2;z_1) 
\end{align*}
with 
\begin{align*}
C_1 = \int_0^1 u^{(\delta_1/2)-1} (1-cu)^{(\delta_2/2)-1} du 
, \quad 
\tilde{q}(z_2;z_1) = \int_0^{z_2} dx_2 A_{3,1} A_{3,2} 
\end{align*}
where $ A_{3,1}=A_{2,1} $ and 
\begin{align*}
A_{3,2} = (x_2)^{1-\delta_1/2} (z_2-cx_2)^{1-\delta_2/2} 
p_{1}^{\delta_1}(0+,x_2) p_{1}^{\delta_2}(0+,z_2-cx_2) . 
\end{align*}
\end{Lem}

\begin{proof}
Recall that 
\begin{align}
q(z_2,z_3;z_1) = \int_0^{z_3} dx_3 A_{2,3}(z_3,x_3) 
\label{prf Bessel: F2}
\end{align}
where 
\begin{align*}
A_{2,3}(z_3,x_3) = \int_0^{z_2} dx_2 A_{3,1} 
p_{1}^{\delta_1}(x_2,x_3) p_{1}^{\delta_2}(z_2-cx_2,z_3-cx_3) . 
\end{align*}
Here we note that $ A_{3,1} $ does not depend on $ z_3 $ nor $ x_3 $. 
If we take $ x_3=z_3 u $ for $ 0<u<1 $, we have 
\begin{align*}
A_{2,3}(z_3,z_3 u) = \int_0^{z_2} dx_2 A_{3,1} 
p_{1}^{\delta_1}(x_2,z_3 u) p_{1}^{\delta_2}(z_2-cx_2,z_3(1-cu)) . 
\end{align*}
Using \eqref{to0}, we have, as $ z_3 \to 0+ $, 
\begin{align*}
z_3^{2-(\delta_1+\delta_2)/2} A_{2,3}(z_3,z_3 u) 
\to u^{(\delta_1/2)-1} (1-cu)^{(\delta_2/2)-1} 
\int_0^{z_2} dx_2 A_{3,1} A_{3,2} . 
\end{align*}
Changing variables to $ u=x_3/z_3 $ in the integral \eqref{prf Bessel: F2}, we obtain 
\begin{equation*}
z_3^{1-(\delta_1+\delta_2)/2} q(z_2,z_3;z_1) 
= z_3^{2-(\delta_1+\delta_2)/2} \int_0^{1} du A_{2,3}(z_3,z_3 u) , 
\end{equation*}
which converges to $ C_1 \tilde{q}(z_2;z_1) $ as $ z_3 \to 0+ $. 
\end{proof}

Fourth, we study the asymptotic behaviors of $ \tilde{q}(z_2;z_1) $ and $ q(z_2;z_1) $ 
as $ z_2 \to \infty $. 
Recall that 
\begin{equation*}
\begin{split}
\tilde{q}(z_2;z_1) 
=& \int_0^{z_2} dx_2 A_{3,1} A_{3,2} \\
=& \int_0^{z_2} dx_2 \, x_2^{1-\delta_1/2} (z_2-cx_2)^{1-\delta_2/2} 
p_{1}^{\delta_1}(0+,x_2) \\
& \times  
p_{1}^{\delta_2}(z_1,z_2-cx_2) 
p_{1}^{\delta_1}(0+,x_2) p_{1}^{\delta_2}(0+,z_2-cx_2) \\
=& z_2^{3-(\delta_1+\delta_2)/2} 
\int_0^{1} du \, u^{1-\delta_1/2} (1-cu)^{1-\delta_2/2} 
p_{1}^{\delta_1}(0+,z_2u) \\
& \times 
p_{1}^{\delta_2}(z_1,z_2(1-cu)) 
p_{1}^{\delta_1}(0+,z_2u) p_{1}^{\delta_2}(0+,z_2(1-cu)) 
\end{split}
\end{equation*}
and that 
\begin{align*}
q(z_2;z_1) 
=& z_2 \int_0^{1} du \, 
p_{1}^{\delta_1}(0+,z_2u) p_{1}^{\delta_2}(z_1,z_2(1-cu)) . 
\end{align*}

\begin{Lem} \label{lem: lemma3}
Let $ r>0 $. Then 
\begin{align}
\frac{\tilde{q}(z_2;z_2r)}{q(z_2;z_2r)} 
\sim C_2 D(r)^{-\delta_{1}/2} e^{-z_2/2} 
\quad \text{as $ z_2 \to \infty $} \label{Lem 3}
\end{align}
where $ C_2 $ is some positive constant 
depending only on $ \delta_1 $ and $ \delta_2 $ and 
\begin{align*}
D(r) = 1 + \frac{1-c}{1-c + \sqrt{r} c} . 
\end{align*}
\end{Lem}

\begin{proof}
If we express $\tilde{q}(z_2;z_2r) $ as 
\begin{align*}
r^{(1-\delta_2)/4} z_2^{(\delta_1 - 1)/2} 
\int_0^{1} f_{1}(z_2,u) e^{-z_2 \phi_{1}(u)} u^{\delta_1/2-1} du 
\end{align*}
using 
\begin{align*}
\phi_{1}(u) = b_{1} u + \sqrt{r} \kbra{1-\sqrt{1-cu}} + a_{1} 
\end{align*}
with $ b_{1} = 1-c $ and $ a_{1} = (\sqrt{r}-1)^2/2 + 1/2 $, 
then $ f_{1}(z_2,\cdot\,) $ turns out to be a bounded continuous function such that 
$ f_{1}(z_2,u/z_2) $ converges to a constant 
depending only on $ \delta_1 $ and $ \delta_2 $ as $z_2 \to \infty$, by \eqref{far}. 
Since $ \phi_{1} $ and $ f_{1} $ satisfies the assumptions, 
we can use Lemma \ref{prf Bessel: asymp} and hence we obtain 
\begin{align}
\tilde{q}(z_2;z_2r) 
\sim C_{2,1} r^{(1-\delta_2)/4} \phi_{1}'(0+)^{-\delta_1/2} z_2^{-1/2} e^{- a_{1} z_2} 
\quad \text{as $ z_2 \to \infty $} 
\label{prf Bessel: asymp1}
\end{align}
with some constant $ C_{2,1} $ 
depending only on $ \delta_1 $ and $ \delta_2 $. 

We also have a similar expression 
\begin{align*}
r^{(1-\delta_2)/4} z_2^{(\delta_1-1)/2} 
\int_0^1 f_{2}(z_2, u) e^{-z_2 \phi_{2}(u)} u^{\delta_1/2-1} du 
\end{align*}
for $q(z_2;z_2r) $ using  
\begin{align*}
\phi_{2}(u) = b_{2} u + \sqrt{r} \kbra{1-\sqrt{1-cu}} + a_{2} 
\end{align*}
with $ b_{2} = (1-c)/2 $ and $ a_{2}=(\sqrt{r}-1)^2/2 $ 
and a function $ f_{2}(z_2,\cdot\,) $ as before. 
Thus the same argument yields
\begin{align}
q(z_2;z_2r) 
\sim C_{2,2} r^{(1-\delta_2)/4} \phi_{2}'(0+)^{-\delta_1/2} z_2^{-1/2} e^{-a_{2} z_2} 
\quad \text{as $ z_2 \to \infty $} 
\label{prf Bessel: asymp2}
\end{align}
with some constant $ C_{2,2} $ 
depending only on $ \delta_1 $ and $ \delta_2 $. 

Using \eqref{prf Bessel: asymp1} and \eqref{prf Bessel: asymp2}
together with 
$ \phi_{1}'(0+) = b_{1} + \sqrt{r} c / 2 $ and 
$ \phi_{2}'(0+) = b_{2} + \sqrt{r} c / 2 $, 
we obtain \eqref{Lem 3}. 
\end{proof}

Now we are in a position to prove Theorem \ref{thm: Bessel}. 

\begin{proof}[Proof of Theorem \ref{thm: Bessel}]
Let $ 0<c<1 $. 
We combine Lemmas \ref{lem: lemma1}, \ref{lem: lemma2} and \ref{lem: lemma3} 
to obtain 
\begin{align*}
\lim_{z_2 \to \infty } e^{z_2/2} \lim_{z_3 \to 0+} z_3^{1-(\delta_1+\delta_2)/2} 
\lim_{\eps \to 0+} 
\frac{q(z_2,z_3;\eps,z_2r)}{q(z_2;\eps,z_2r)} 
= C_3 D(r)^{-\delta_{1}/2}  
\end{align*}
for some constant $ C_3 $ which depends only on $ \delta_1 $, $ \delta_2 $ and $ c $. 
Therefore we conclude that 
the conditional probability \eqref{prf Bes: cond trans prob} 
does depend on $ (\eps,z_1) $, 
which proves that $ Z^c $ is non-Markov. 
\end{proof}

\paragraph{
\textbf{Acknowledgements.}
The authors thank Professors Hideki Tanemura and Makoto Katori
for helpful discussions. They also thank Professor 
Tomoyuki Shirai for drawing their attention to \cite{DE02}. 
}


\end{document}